\documentclass[a4paper,11pt]{article}
\usepackage{mathrsfs}
\usepackage{amsfonts}
\usepackage{amsfonts}
\usepackage{amssymb}
\usepackage{mathrsfs}
\usepackage{indentfirst,latexsym,bm,amsmath,amssymb,amsthm}
\usepackage{xcolor}
\usepackage[
                   bookmarksnumbered=true,
                   bookmarksopen=true, 
                   pdfauthor=kmc,
                   pdfcreator={LaTex with hyperref package + WinEdt},
                   pdftitle=trudge,
                   colorlinks,%
                   citecolor=blue,
                   linkcolor=blue,%
                   urlcolor=blue,
                   hyperindex,%
                   plainpages=false,%
                   pdfstartview=FitH,
                   linktocpage=true,
                   dvipdfm%
                   ]{hyperref}
\usepackage{indentfirst,latexsym,bm,amsmath,amssymb,amsthm}
\usepackage[dvips]{graphicx}

\makeatletter\@addtoreset{equation}{section} \makeatother
\newtheorem{thm}{Theorem}[section]

\newtheorem{rem}{Remark}[section]
\newtheorem{defn}{Definition}[section]
 

\title{ Remarks on Nonlinear Elastic Waves in the Radial Symmetry in 2-D}

\author{ Dongbing Zha\thanks{ School of Mathematical Sciences, Fudan University, Shanghai 200433, PR China.{ E-mail address: ZhaDongbing@fudan.edu.cn}.     }}

\begin{document}

\maketitle
\begin{abstract}
In this manuscript we first give the explicit variational structure of the nonlinear elastic waves for isotropic, homogeneous, hyperelastic materials in 2-D. Based on this variational structure, we suggest a null condition which is a kind of structural condition
on the nonlinearity in order to stop the formation of finite time singularities
of local smooth solutions. In the radial symmetric case, inspired by Alinhac's work \cite{Alinhac01} on 2-D quasilinear wave equations, we show that such null condition can ensure the global existence of smooth solutions with small initial data.\\
\emph{keywords}: Nonlinear elastic waves; 2-D; variational structure; null condition; global existence.
\end{abstract}
\pagestyle{plain} \pagenumbering{arabic}

\section{ Introduction  }
For elastic materials, the motion for the displacement is governed by the nonlinear elastic wave equations which is a second-order quasilinear hyperbolic system. For isotropic, homogeneous, hyperelastic materials, the motion for the displacement $u=u(t,x)$ satisfies
\begin{align}\label{elastic}
\partial_t^2u-c_2^2\Delta u-(c_1^2-c_2^2)\nabla\nabla\cdot u=N(\nabla u,\nabla^2u),
 \end{align}
 where the nonlinear term $N(\nabla u,\nabla^2u)$ is linear in $\nabla^2u$. Some physical backgrounds of the nonlinear elastic waves can be found in
 Ciarlet \cite{MR936420} and Gurtin \cite{Gurtin81}. Here the main concern for us is the problem of long time existence of smooth solutions for \eqref{elastic}, which can trace back to Fritz John's pioneering work on elastodynamics(see~Klainerman \cite{Klainerman98}).
\par
In the 3-D case, John \cite{John84} proved that in the radial
symmetric case, a genuine nonlinearity condition will lead to the formation of
singularities for small initial data. John also \cite{John88} showed that the equations have almost global smooth solutions
for small initial data(see also a simplified proof in Klainerman and~Sideris \cite{Klainerman96}). Then~Sideris \cite{Sideris00}
proved that for certain classes of materials that satisfy a null condition, there exist global smooth solutions with small initial data(see also previous result in Sideris \cite{Sideris96}). Agemi \cite{Agemi00} also suggested a kind of null condition and established the global existence result under such null condition independently. The null condition suggested in Agemi \cite{Agemi00} is equivalent with the one in~Sideris \cite{Sideris00}, and is the complement of the genuine nonlinearity condition given by John \cite{John84}(see Sideris \cite{Sideris00} and Xin \cite{MR1930184}). For large initial data,  Tahvildar-Zadeh \cite{MR1648985} proved that singularities will always form no matter whether the null condition holds or not.
 \par
 In this manuscript, we will consider the 2-D case. The objective of this manuscript is twofold. The first one is to give the explicit variational structure of the nonlinear elastic waves in 2-D, and suggest a null condition on the nonlinearity based on the variational structure. The other is to prove, for radial symmetric and small initial data, such null condition can ensure the global existence of smooth solutions of the Cauchy problem for nonlinear elastic waves in 2-D. To achieve this goal, we will use the global existence result of Alinhac on 2-D quasilinear wave equations with null condition in \cite{Alinhac01}.
 \par
 An outline of this paper is as follows. The main theorem on global existence is stated and proved in Sect. 4, after characterization
of the nonlinear term by the null condition in Sect. 3.
The derivation of null condition is based on the variational structure of the nonlinear elastic waves in 2-D which is given in Sect. 2. Some related remarks are given in Sect. 5.

\section{Nonlinear elastic waves in 2-D}
Consider a homogeneous elastic material
filling in the whole space $\mathbb{R}^2$. Assume that its density in its undeformed state is unity. Let $y: \mathbb{R}\times \mathbb{R}^{2}\longrightarrow \mathbb{R}^{2}$ be the smooth deformation of the material that evolves with time, which is an orientation preserving
diffeomorphism taking a material point $x\in \mathbb{R}^2$ in the reference configuration
to its position $y(t,x)\in \mathbb{R}^2$ at time $t$. The deformation gradient is then the matrix $F=\nabla y$ with components
$F_{il}=\partial_{l}y^{i}$, where the spatial gradient will be denoted by $\nabla$.\par
For the
materials under consideration, the potential energy density is characterized
by a stored energy function $W(F)$. Then we have the Lagrangian
 \begin{align}\label{lagrange}
 \mathscr{L}(y)=\iint \frac{1}{2}|y_t|^2-W(\nabla y)~ dxdt.
\end{align}
A material is frame indifferent,
respectively, isotropic if the conditions
\begin{align}\label{frame}
W(F)=W(QF)~~\text{and}~~ W(F)=W(FQ)~~
\end{align}
hold for every orthogonal matrix $Q$.
It is well-known that \eqref{frame} implies
that the stored energy function $W(F)=\overline{\sigma}(\iota_1,\iota_2)$, where~$\iota_1,\iota_2$ are principal
invariants of the (left) Cauchy--Green strain tensor~$FF^{T}$.
By applying Hamilton's principle to \eqref{lagrange}, we can get the corresponding Euler--Lagrange equation
\footnote{Repeated indices are always summed.}
\begin{align}\label{waver}
\frac{\partial^2 y^i}{\partial t^2}-\frac{\partial}{\partial x^{l}}\big(\frac{\partial W}{\partial F_{il}}(\nabla y)\big)=0.
\end{align}
\par
We will consider displacements $u(t,x)=y(t,x)-x$ from the
reference configuration. The displacement gradient matrix $G=\nabla u$ satisfies $G=F-I$, and~$C=FF^{T}-I=G+G^{T}+GG^{T}$. Consequently we have
\begin{align}
 W(F)={\sigma}(k_1, k_2),
 \end{align}
where~$k_1, k_2$ are principal
invariants of $C$. For the displacement, we have the~Lagrangian
\begin{align}\label{lagrange2}
\widetilde{\mathscr{L}}(u)= \mathscr{L}(y)=\iint \frac{1}{2}|u_t|^2-{\sigma}(k_1, k_2)~ dxdt.
 \end{align}
Then the PDE's can be formulated as the nonlinear system
\begin{align}\label{waverr}
\frac{\partial^2 u^i}{\partial t^2}-\frac{\partial}{\partial x^{l}}\frac{\partial \sigma}{\partial G_{il}}=0.
 \end{align}
 \par
 Now in order to give the variational structural of the nonlinear elastic waves, we need to represent~${\sigma}(\kappa_1,\kappa_2)$ by~$G=\nabla u$ explicitly. We will consider only small displacements from the
reference configuration. In two space dimensions, the global existence of small
amplitude solutions to nonlinear hyperbolic systems hinges on the specific form
of the quadratic and cubic portion of the nonlinearity in relation to the linear part(see for example, Alinhac \cite{Alinhac01}). Such
compatibility conditions are often referred to as null conditions(see Sect. 3). From the analytical point of view, therefore, it is enough to truncate \eqref{waverr} at fourth order in $u$, the higher order corrections having no influence on the existence of small
solutions. And we will truncate ${\sigma}(k_1, k_2)$ in \eqref{lagrange2} at fifth order in $u$.\par
Let~$\lambda_1,\lambda_2$ be the  eigenvalues of $C$. We use the following formula for principal invariants:
\begin{align}
 &k_1=\lambda_1+\lambda_2=\text{tr}~ C,\\
 &k_2=\lambda_1\lambda_2=\text{det}~ C=\frac{(\text{tr}~ C)^2-\text{tr}~ C^2}{2}.
 \end{align}
Noting that¡¡
\begin{align}
 \text{tr}~ C&=2\text{tr}~ G+\text{tr}~ GG^{T},\\
 (\text{tr}~ C)^2&=4(\text{tr}~ G)^2+4\text{tr}~ G~\text{tr}~ GG^{T}+(\text{tr}~ GG^{T})^2,\\
 \text{tr}~ C^2&=2(\text{tr}~ G^2+\text{tr}~ GG^{T})+4\text{tr}~ G^2G^{T}+\text{tr}~ (GG^{T})^2,
 \end{align}
we see that
 \begin{align}\label{O}
 k_1&=2\text{tr}~ G+\text{tr}~ GG^{T},\\\label{O1}
k_2&=2(\text{tr}~ G)^2-(\text{tr}~ G^2+\text{tr}~ GG^{T})
+2(\text{tr}~ G~\text{tr}~ GG^{T}-\text{tr}~ G^2G^{T})\nonumber\\
&+\frac{1}{2}\big((\text{tr}~ GG^{T})^2-\text{tr}~ (GG^{T})^2\big).
 \end{align}
 From \eqref{O} and \eqref{O1} it is apparent that $k_1=\mathscr{O}(|G|), k_2=\mathscr{O}(|G|^2)$.
Therefore, the relevant terms in the Taylor expansion of $k_i=0$ are
 \begin{align}\label{zhan}
 \sigma(k_1,k_2)&=(\sigma_0+\sigma_1 k_1)
 +(\frac{1}{2}\sigma_{11}k_1^2+\sigma_2k_2)+(\frac{1}{6}\sigma_{111}k_1^3+\sigma_{12}k_1k_2)\nonumber\\
 &
 +(\frac{1}{24}\sigma_{1111}k_1^4+\frac{1}{2}\sigma_{112}k_1^2k_2+\frac{1}{2}\sigma_{22}k_2^2)+h.o.t.,
 \end{align}
 with $h.o.t.$ denoting higher order terms, and the constants $\sigma_0,\sigma_1$ etc., standing for the partial derivatives of $ \sigma$ at $k_i=0$. Without loss of generality, we assume that $\sigma_0=0$. And we impose the condition $\sigma_1=0$, which implies that the reference configuration
is a stress-free state. Denote
 \begin{align}\label{zhan1}
\sigma(k_1,k_2)&=l_2(G)+l_3(G)+l_4(G)+\mathscr{O}(|G|^5),
\end{align}
where~$l_{i}(G)(i=2,3,4)$ represents the~homogeneous $i-$th order part of~$ \sigma(k_1,k_2)$ with respect to~$G=\nabla u$.
By \eqref{O} and \eqref{O1},  after a bit of calculation, we see that
 \begin{align}\label{l2}
 l_2(G)&=2(\sigma_{11}+\sigma_2)(\text{tr}~ G)^2-\sigma_{2}(\text{tr}~ G^2+\text{tr}~ GG^{T}),\\\label{l3}
 l_3(G)&=2(\sigma_{11}-\sigma_{12})\text{tr}~ G~\text{tr}~ GG^{T}+2\sigma_2 (\text{tr}~ G~\text{tr}~ GG^{T}-\text{tr}~ G^2G^{T})\nonumber\\
 &+(\frac{4}{3}\sigma_{111}+4\sigma_{12})(\text{tr}~ G)^3-2\sigma_{12}\text{tr}~ G ~\text{tr}~ G^2,\\\label{l4}
 l_4(G)&=\frac{1}{2}\sigma_{11}(\text{tr}~ GG^{T})^2+\frac{1}{2}\sigma_{2}\big((\text{tr}~ GG^{T})^2-\text{tr}~ (GG^{T})^2\big)\nonumber\\
 &+2(\sigma_{111}+\sigma_{12})(\text{tr}~ G)^2~\text{tr}~ GG^{T}
+4\sigma_{12}(\text{tr}~ G)\big(\text{tr}~ G~\text{tr}~ GG^{T}-\text{tr}~ G^2G^{T} )\nonumber\\
 &-\sigma_{12}\text{tr}~ GG^{T}(\text{tr}~ GG^{T}+\text{tr}~ G^2)
+(\frac{2}{3}\sigma_{1111}+4\sigma_{112}+2\sigma_{22})(\text{tr}~ G)^4\nonumber\\
 &-2(\sigma_{112}+\sigma_{22})(\text{tr}~ G)^2(\text{tr}~ G^2+\text{tr}~ GG^{T})
+\frac{1}{2}\sigma_{22}(\text{tr}~ G^2{\color{blue}{+}} \text{tr}~ GG^{T})^2.
\end{align}
\par
Our task now is to represent~$l_{i}(G)(i=2,3,4)$ via $G=\nabla u$ explicitly. Denote the null forms
\begin{align}
Q_{ij}(v,w)=\partial_i v\partial_j w-\partial_i w \partial_{j} v,~1\leq i,j\leq 2.
\end{align}
First it is easy to see that
\begin{align}\label{l21}
\text{tr}~ G&=\nabla \cdot u,\\\label{l22}
\text{tr}~ G^2&=(\nabla \cdot u)^2-2Q_{12}(u^1,u^2),\\\label{l23}
\text{tr}~ GG^{T}&=|\nabla u|^2=(\nabla \cdot u)^2+(\nabla^{\bot} \cdot u)^2-2Q_{12}(u^1,u^2),
\end{align}
 where~$\nabla^{\bot }=(\partial_2, -\partial_1)$.
It follows from~\eqref{l21}, \eqref{l22} and~\eqref{l23} that
 \begin{align}\label{L2}
 l_2(\nabla u)=(2\sigma_{11}+\sigma_2)(\nabla \cdot u)^2-\sigma_{2}|\nabla u|^2+2\sigma_{2}Q_{12}(u^1,u^2).
 \end{align}
Next we compute~$l_3(\nabla u)$. According to~\eqref{l21} and~\eqref{l23}, we can get
\begin{align}\label{l2221}
\text{tr}~ G~ \text{tr}~ GG^{T}=(\nabla \cdot u)^3+(\nabla \cdot u)(\nabla^{\bot} \cdot u)^2-2(\nabla \cdot u) Q_{12}(u^1,u^2).
 \end{align}
 We can also show that
 \begin{align}\label{curll}
 \text{tr}~ G~\text{tr}~ GG^{T}-\text{tr}~ G^2G^{T}
 &=\partial_i u^{i}\partial_k u^{j}\partial_k u^{j}-\partial_i
  u^{j}\partial_{k}u^{i}\partial_{k}u^{j}
  =\partial_k u^{j}(\partial_i u^{i}\partial_k u^{j}-\partial_i
  u^{j}\partial_{k}u^{i})\nonumber\\
 & =\partial_k u^{j}Q_{ik}(u^{i},u^{j})
  =(\nabla\cdot u)~ Q_{12}(u^1,u^2).
 \end{align}
 Since~\eqref{l21} and~\eqref{l22}, it follows that
 \begin{align}\label{l22212}
 (\text{tr}~ G)^3&=(\nabla \cdot u)^3,\\\label{yi}
\text{tr}~ G ~\text{tr}~ G^2&=(\nabla \cdot u)^3-2(\nabla \cdot u) Q_{12}(u^1,u^2).
 \end{align}
So it is a consequence of~\eqref{l2221}--\eqref{yi} that
 \begin{align}\label{sanjie}
 l_3(\nabla u)&=d_1(\nabla \cdot u)^3+d_2(\nabla \cdot u)(\nabla^{\bot} \cdot u)^2+d_3(\nabla \cdot u)Q_{12}(u^1,u^2),
 \end{align}
 where
\begin{align}
 \left \{
\begin{array}{llll}
d_1=2\sigma_{11}+\frac{4}{3}\sigma_{111},\\
 d_2=2(\sigma_{11}-\sigma_{12}),\\
 d_3=2(-2\sigma_{11}+4\sigma_{12}+\sigma_2).
\end{array} \right.
 \end{align}
Finally we consider~$ l_4(\nabla u)$. By~\eqref{l23}, we get that
 \begin{align}\label{l22222}
 (\text{tr}~ GG^{T})^2
 =& (\nabla \cdot u)^4+(\nabla^{\bot} \cdot u)^4+4Q^2_{12}(u^1,u^2)+2(\nabla \cdot u)^2(\nabla^{\bot} \cdot u)^2\nonumber\\
 &-4(\nabla \cdot u)^2Q_{12}(u^1,u^2)-4(\nabla^{\bot} \cdot u)^2Q_{12}(u^1,u^2).
 \end{align}
It can be shown that
 \begin{align}\label{L42}
&(\text{tr}~ GG^{T})^2-\text{tr}~ (GG^{T})^2
=\partial_j u^{i}\partial_j u^{i}\partial_l u^{k }\partial_l u^{k }-\partial_j u^{i}\partial_j u^{k}\partial_l u^{k}\partial_{l}u^{i}\nonumber\\
&=\partial_j u^{i}\partial_l u^{k}(\partial_j u^{i}\partial_l u^{k }-\partial_j u^{k}\partial_{l}u^{i})
=\partial_j u^{i}\partial_l u^{k} Q_{jl}(u^{i},u^{k})\nonumber\\
&=2Q_{12}^{2}(u^{1},u^{2}).
 \end{align}
It follows from~\eqref{l21} and~\eqref{l23} that
 \begin{align}\label{L43}
 (\text{tr}~ G)^2~\text{tr}~ GG^{T}
 = (\nabla \cdot u)^4+(\nabla \cdot u)^2(\nabla^{\bot} \cdot u)^2-2(\nabla \cdot u)^2 Q_{12}(u^1,u^2).
 \end{align}
Due to~\eqref{l21} and~\eqref{curll}, we can see that
 \begin{align}\label{L44}
 (\text{tr}~ G)\big(\text{tr}~ G~\text{tr}~ GG^{T}-\text{tr}~ G^2G^{T} )=(\nabla\cdot u)^2~ Q_{12}(u^1,u^2).
 \end{align}
 By~\eqref{l21},\eqref{l22} and~\eqref{l23}, we have that
 \begin{align}\label{L45}
 (\text{tr}~ G^2)~\text{tr}~ GG^{T}
 &= (\nabla \cdot u)^4+4Q^2_{12}(u^1,u^2)+(\nabla \cdot u)^2(\nabla^{\bot} \cdot u)^2\nonumber\\
 &~~-4(\nabla \cdot u)^2Q_{12}(u^1,u^2)-2(\nabla^{\bot} \cdot u)^2Q_{12}(u^1,u^2),\\\label{L46}
 (\text{tr}~ G)^2~\text{tr}~ G^2
 &= (\nabla \cdot u)^4-2(\nabla \cdot u)^2Q_{12}(u^1,u^2),\\\label{L47}
(\text{tr}~ G^2)^2&=(\nabla \cdot u)^4+4Q^2_{12}(u^1,u^2)-4(\nabla \cdot u)^2Q_{12}(u^1,u^2).
 \end{align}
 So it is a consequence of~~\eqref{l22222}--\eqref{L47} that
 \begin{align}\label{sijie}
 l_4(\nabla u)
 &=e_1(\nabla\cdot u)^4+e_2(\nabla^{\bot} \cdot u)^4+e_3(\nabla \cdot u)^2(\nabla^{\bot} \cdot u)^2+e_4Q^2_{12}(u^1,u^2)\nonumber\\
 &~~~+e_5(\nabla \cdot u)^2Q_{12}(u^1,u^2)+e_6(\nabla^{\bot} \cdot u)^2Q_{12}(u^1,u^2),
 \end{align}
 where
 \begin{align}
 \left \{
\begin{array}{llll}
e_1=\frac{1}{2}\sigma_{11}+2\sigma_{111}+\frac{2}{3}\sigma_{1111},\\
e_2=\frac{1}{2}\sigma_{11}-\sigma_{12}+\frac{1}{2}\sigma_{22},\\
e_3=\sigma_{11}-\sigma_{12}+2\sigma_{111}-2\sigma_{112},\\
e_4=2\sigma_{11}-8\sigma_{12}+8\sigma_{22}+\sigma_2,\\
e_5=2(-\sigma_{11}+4\sigma_{12}-2\sigma_{111}+4\sigma_{112}),\\
e_6=2(-\sigma_{11}+3\sigma_{12}-2\sigma_{22}).
\end{array} \right.
 \end{align}
 \par
From~\eqref{lagrange2}, \eqref{zhan1}, \eqref{L2}, \eqref{sanjie}, \eqref{sijie}, by Hamilton's principle we get the nonlinear elastic wave equations in 2-D:
\begin{align}\label{fei}
\partial^2_t u-c_2^2 \triangle u-(c_1^2-c_2^2)\nabla \nabla\cdot u =N_2(\nabla u, \nabla^2 u)+N_3(\nabla u, \nabla^2 u).
\end{align}
The material constants $c_1$ and $c_2$ ($c_1 > c_2>0$) correspond to the speed of pressure wave and shear wave, respectively. We also have
\begin{align}
c_2^2=-2 \sigma_2, c_1^2=4 \sigma_{11}.
\end{align}
The quadratic term
\begin{align}\label{N2}
N_2(\nabla u, \nabla^2 u)&=3d_1\nabla (\nabla\cdot u)^2+d_2\big( \nabla (\nabla^{\bot} \cdot u)^2+2\nabla^{\bot}(\nabla\cdot u~ \nabla^{\bot} \cdot u) \big)\nonumber\\
&+Q( u, \nabla u),
\end{align}
where
\begin{align}
Q( u, \nabla u)=d_3\nabla Q_{12}( u^{1}, u^{2})+d_3\big(Q_{12}(\nabla\cdot u, u^{2}),Q_{12}(u^{1}, \nabla\cdot u) \big).
\end{align}
And the cubic term
\begin{align}\label{N3}
N_3(\nabla u, \nabla^2 u)&=4e_1 \nabla (\nabla\cdot u)^3+4e_2\nabla^{\bot}(\nabla^{\bot} \cdot u)^3\nonumber\\
&+2e_3\big( \nabla( (\nabla\cdot u)(\nabla^{\bot} \cdot u)^2 )+\nabla^{\bot}(
 (\nabla^{\bot}\cdot u)(\nabla \cdot u)^2 )\big )\nonumber\\
 &+\widetilde{{Q}}( u, \nabla u),
\end{align}
where
\begin{align}\label{N4}
\widetilde{{Q}}( u, \nabla u)&=2e_4\Big(Q_{12}\big(Q_{12}(u^1,u^2), u^2\big),Q_{12}\big(u^1,Q_{12}(u^1,u^2)\big) \Big)\nonumber\\
&+2e_5\nabla\big((\nabla\cdot u)Q_{12}(u^1,u^2)\big)\nonumber\\
&+e_5\Big(Q_{12}\big((\nabla\cdot u)^2, u^{2}\big), Q_{12}\big(u^{1},(\nabla\cdot u)^2 \big)\Big)\nonumber\\
&+2e_6\nabla^{\bot}\big((\nabla^{\bot}\cdot u)Q_{12}(u^1,u^2)\big)\nonumber\\
&+e_6\Big(Q_{12}\big((\nabla^{\bot}\cdot u)^2, u^{2}\big),Q_{12}\big(u^{1},(\nabla^{\bot}\cdot u)^2 \big)\Big).
\end{align}
\par

\begin{rem}\label{remmm}
The nonlinear terms in the equation \eqref{fei} can be also represented as follows. For the quadratic term,
\begin{align}
N^{(i)}_2(\nabla u, \nabla^2 u)=B^{ijk}_{lmn}\partial_{l}(\partial_m u^{j}\partial_nu^{k}), i=1,2,
 \end{align}
 and for the cubic term,
\begin{align}
N^{(i)}_3(\nabla u, \nabla^2 u)=B^{ijkp}_{lmnq}\partial_{l}(\partial_m u^{j}\partial_nu^{k}\partial_qu^{p}),i=1,2.
 \end{align}
Here
\begin{align}
 B^{ijk}_{lmn}=\frac{1}{2}\frac{\partial^3W}{\partial F_{il}\partial F_{jm}\partial F_{kn}}(I),~~
B^{ijkp}_{lmnq}=\frac{1}{6}\frac{\partial^4W}{\partial F_{il}\partial F_{jm}\partial F_{kn}\partial F_{pq}}(I),
\end{align}
and the following symmetry condition holds
 \begin{align}\label{sym1232}
B^{ijk}_{lmn}=B^{jik}_{mln}=B^{kji}_{nml},~~
B^{ijkp}_{lmnq}=B^{jikp}_{mlnq}= B^{kjip}_{nmlq}= B^{pjki}_{qmnl}.
 \end{align}
 We can also know that~$\{B^{ijk}_{lmn}\}$ is an isotropic six-order tensor and $\{B^{ijkp}_{lmnq}\}$ is an isotropic eight-order tensor thanks to~\eqref{frame}.
\end{rem}
\section{The null condition}
For quasilinear hyperbolic systems such as the nonlinear elastic wave equations, local smooth solution in general will develop singularities such as shock waves even for small enough initial data. So a nature problem is if we can put some structural condition on the
nonlinearity to ensure the global existence of smooth solution at least for small initial data.  The pioneering work in this aspect
belongs to Sergiu Klainerman. In Klainerman \cite{Klainerman82}, for quasilinear wave equations he identified such structural condition which is called \lq\lq null condition". Under such null condition, the global existence of smooth solutions of 3-D quasilinear wave equations was proved by Christodoulou \cite{Christodoulou86} and~Klainerman \cite{Klainerman86} independently. It is worth noting that in the 3-D case, the time decay of the linear system is $(1+t)^{-1}$, so we should only put the null condition on the quadratic term in the equation. In the 2-D case, since the slow time decay $(1+t)^{-\frac{1}{2}}$ of the linear system, we should put the null condition not only on the quadratic but also on the cubic term in the equation.
The 2-D case is more difficult and was solved in Alinhac \cite{Alinhac01}. For some earlier results in 2-D case, we refer the reader to Godin \cite{MR1218523}, Hoshiga \cite{MR1325960} and Katayama \cite{MR1256442}. Some different concepts of null condition can be found in John \cite{John90}, H\"{o}rmander \cite{Hormander2} and Alinhac \cite{MR1339762}. For 3-D nonlinear elastic waves, Agemi \cite{Agemi00} and~Sideris \cite{Sideris00} suggested the corresponding null condition, and proved the global existence of small smooth solutions.
 \par
 In the remainder of this section, for nonlinear systems with variational structure we will give a new kind of null condition which was first suggested
 in Zhou \cite{Zhou}. Then for the nonlinear elastic waves in 2-D, the corresponding null condition will be derived.
\par
 Suppose that the nonlinear system under consideration admits a variational
structure:
  \begin{align}\label{Lag}
 \mathscr{L}(\phi)=\iint l(\phi,\partial \phi)dxdt,
 \end{align}
where~$\mathscr{L}$ is the Lagrangian, and $l$ is the Lagrangian density. Then the nonlinear system is the Euler-Lagrangian equation of
\eqref{Lag}:
 \begin{align}\label{nonlinear}
 F(\phi,\partial \phi,\partial^2 \phi)=0.
   \end{align}
For smooth~$l$ and~$F$, by the~Taylor expansion we have that near the origin,
\begin{align}
l(\xi)&=l_2(\xi)+l_{3}(\xi)+l_{4}(\xi)+\mathscr{O}(|\xi|^{5}),\\
F(\eta)&=F_1(\eta)+ F_2(\eta)+F_3(\eta)+ \mathscr{O}(|\eta|^{4}). \end{align}
It is easy to see that
 for $ i=1,2,3$,
\begin{align}
F_i(\phi,\partial \phi,\partial^2 \phi)=0
\end{align}
 is the Euler--Lagrange equation of
 \begin{align}
 \mathscr{L}_{i+1}(\phi)= \iint l_{i+1}(\phi,\partial \phi)dxdt.
 \end{align}
 Consider the plane wave solutions of the linearized equation~$F_1(\phi,\partial \phi,\partial^2 \phi)=0$. Denote by~$\mathscr{P}$ the set of all plane wave solutions, i.e.,
 \begin{align}\label{plane}
\mathscr{P}=&\Big\{\phi(t,x)=\varphi(at+\omega\cdot x):~F_1(\phi,\partial \phi,\partial^2 \phi)=0, \nonumber\\
 &~\varphi(0)=0, \varphi'(0)=0, a\in \mathbb{R}, |\omega|=1 \Big\}.
  \end{align}
We give the following concept of null condition of \eqref{nonlinear}.
 \begin{defn}\label{DEFF}
We say that the nonlinear system~\eqref{nonlinear} satisfies the first null condition, if
\begin{align}
l_3(\phi,\partial \phi)=0,~~\forall~\phi\in \mathscr{P};
\end{align}
and \eqref{nonlinear} satisfies the second null condition, if
\begin{align}\label{nullnull2}
l_{4}(\phi,\partial \phi)=0,~~\forall~\phi\in \mathscr{P}.
\end{align}
\end{defn}
\par
Now for nonlinear elastic waves in 2-D~\eqref{fei}, we will derive the null condition according to Definition \ref{DEFF}.
First we must give all plane wave solutions of the linearized system
\begin{align}\label{linear}
\partial^2_t u-c_2^2 \triangle u-(c_1^2-c_2^2)\nabla \nabla\cdot u=0.
\end{align}
 Suppose that~$u(t,x)=\varphi(at+\omega\cdot x)( a\in \mathbb{R},\omega\in \mathbb{R}^{2}, |\omega|=1, \varphi(0)=0, \varphi'(0)=0)$
is a plane wave solution of~\eqref{linear}, then~$\varphi$ satisfies
 \begin{align}\label{planew}
 (a^2-c_2^2)\varphi''-(c_1^2-c_2^2)\omega(\omega\cdot \varphi'')=0.
 \end{align}
  We can verify that \eqref{planew} is equivalent to
 \begin{align}
 (a^2-c_1^2)\omega\cdot \varphi''=0,
 \end{align}
 and
\begin{align}
 (a^2-c_2^2)\omega^{\bot}\cdot \varphi''=0.
 \end{align}
 Because it can not hold that $\omega\cdot \varphi''$ and $\omega^{\bot}\cdot \varphi''=0$ at the same time(otherwise, we must have $\varphi''=0$, in view of the conditions $\varphi(0)=0, \varphi'(0)=0$, then $\varphi=0$). We have
 \begin{align}
 a^2-c_1^2=0,~\omega^{\bot}\cdot \varphi''=0,
 \end{align}
 or
 \begin{align}
 a^2-c_2^2=0,~\omega\cdot \varphi''=0.
 \end{align}
In the first case, we have $\varphi(at+\omega\cdot x)=\omega \psi_1(c_1t+\omega\cdot x)$, where $\psi_1$ is a scalar function. Similarly, in the second case, we have $\varphi(at+\omega\cdot x)=\omega^{\bot} \psi_2(c_2t+\omega\cdot x)$, where $\psi_2$ is a scalar function. By the above discussion, we know that~\eqref{linear} admits two families of planar waves:
\begin{align}\label{planee}
&\mathcal {W}_{1}(\omega)=\{\omega \psi_1(c_1t+\omega\cdot x): \psi_1~\text{is a scalar function}\},\\\label{plane}
&\mathcal {W}_{2}(\omega)=\{\omega^{\perp} \psi_2(c_2t+\omega\cdot x): \psi_2~\text{is a scalar function}\}.
\end{align}
So the set of plane waves:
\begin{align}
\mathscr{P}=\{\mathcal {W}_{1}(\omega), \mathcal {W}_{2}(\omega)\}.
\end{align}
By Definition \ref{DEFF}, for nonlinear elastic waves in 2-D, the first null condition is
\begin{align}\label{nul1}
l_{3}(\nabla u)=0,~~\forall ~u\in \mathscr{P},
\end{align}
where~$l_3(\nabla u)$ is given by~\eqref{sanjie};
and the second null condition is
\begin{align}\label{nul2}
l_{4}(\nabla u)=0,~~\forall~ u\in \mathscr{P},
\end{align}
where~$l_4(\nabla u)$ is defined in~\eqref{sijie}.
It is easy to verify that~\eqref{nul1} is equivalent to
 \begin{align}
 \frac{3}{2}d_1=3\sigma_{11}+2\sigma_{111}=0;
 \end{align}
and \eqref{nul2} is equivalent to
\begin{align}
6e_1&=3\sigma_{11}+12\sigma_{111}+4\sigma_{1111}=0,\\
2e_2&=\sigma_{11}-2\sigma_{12}+\sigma_{22}=0.
\end{align}
In the later of this manuscript, for nonlinear elastic waves in 2-D, we call that $d_1=0$ is the first null condition and $e_1=e_2=0$ is the second null condition . \par
\begin{rem}
Corresponding to Remark~\ref{remmm}, we can show that~the first null condition $d_1=0$ is equivalent to
\begin{align} \label{null1100}
B^{ijk}_{lmn}\omega_i\omega_j\omega_k\omega_l\omega_m\omega_n&=0,~~\forall ~\omega\in S^{1};
\end{align}
for the second null condition,~$e_1=0$ is equivalent to
\begin{align}\label{nulln1101}
B^{ijkp}_{lmnq}\omega_i\omega_j\omega_k\omega_p\omega_l\omega_m\omega_n\omega_q&=0,~~\forall ~\omega\in S^{1},
\end{align}
and~$e_2=0$ is equivalent to
\begin{align}\label{nulln2222}
B^{ijkp}_{lmnq}\omega^{\perp}_i\omega^{\perp}_j\omega^{\perp}_k\omega^{\perp}_p\omega_l\omega_m\omega_n\omega_q&=0,~~\forall ~\omega\in S^{1}.
\end{align}
The proof of these equivalence is given in Peng and~Zha \cite{Zha} following the 3-D analogue in Sideris \cite{Sideris00}. In \cite{Zha}, in various situations, we get the lifespan of classical solutions for nonlinear elastic waves when the radial symmetry of initial data is not assumed.

\end{rem}
\section{Main theorem and its proof}
In this section, for the Cauchy problem of nonlinear elastic wave equations in 2-D, we will show that under the first null condition $d_1=0$ and the second null condition $e_1=e_2=0$, global existence of smooth solutions with small and radial symmetric initial data can be obtained. The key observation in the proof is that in the radial symmetric case, there exists only the pressure waves. Then the nonlinear elastic waves  reduces to a quasilinear wave system with single wave speed $c_1$, and the corresponding null condition can be deduced from the one of nonlinear elastic waves. So we can apply the global existence result of 2-D quasilinear wave equations with null condition in Alinhac \cite{Alinhac01}(see also alternative proofs in Hoshiga \cite{MR2297944} and Zha \cite{zha123}).
 \par
 For the convenience of applications in the later, we first introduce Alinhac's result.
 In~\cite{Alinhac01}, for the Cauchy problem of 2-D quasilinear wave equations with first and second null conditions,  Alinhac proved the global existence of smooth solutions by the so called~\lq\lq ghost weight" energy estimates.
This result can be extended parallel to quasilinear wave systems with single wave speed.
We describe this result in the case of the nonlinearity contains only spatial derivatives and be of divergence form just corresponding to the situation in our application.
\par
Consider the Cauchy problem of 2-D quasilinear wave systems:
  \begin{align}\label{quasilinear2}
&\partial_t^2 v^i-c^2\Delta v^{i}=g^{ijk}_{lmn}\partial_{l}(\partial_{m} v^{j}\partial_{n} v^{k})+h^{ijkp}_{lmnq}\partial_{l}(\partial_{m}v^{j}\partial_{n}v^{k}\partial_{q} v^{p}),\\\label{quasi}
&t=0:~v^i=\varepsilon f^{i}, v_t^i=\varepsilon g^{i},~1\leq i\leq m.
\end{align}
Assume that the coefficients in the nonlinearity satisfy the following symmetry conditions:
\begin{align}\label{sym1}
g^{ijk}_{lmn}=g^{jik}_{mln}=g^{kji}_{nml},~~h^{ijkp}_{lmnq}=h^{jikp}_{mlnq}=h^{kjip}_{nmlq}=h^{pjki}_{qmnl}.
\end{align}
We call that~\eqref{quasilinear2} satisfies the first null condition, if
\begin{align}\label{null1}
g^{ijk}_{lmn}\omega_l\omega_m\omega_n=0,~\forall~1\leq i,j,k\leq m, \omega\in S^1;
\end{align}
and \eqref{quasilinear2} satisfies the second null condition, if
\begin{align}\label{null222}
h^{ijkp}_{lmnq}\omega_l\omega_m\omega_n\omega_q=0,~\forall~1\leq i,j,k,p\leq m, \omega\in S^1.
\end{align}
\par
The following theorem is obtained implicitly in Alinhac \cite{Alinhac01}.
\begin{thm}\label{thm41}
Consider the Cauchy problem~\eqref{quasilinear2}--\eqref{quasi}.
Assume that~\eqref{quasilinear2} satisfies the first null condition~\eqref{null1} and
the second null condition~\eqref{null222}, and the initial data~$f,g$ is smooth and has compact support. Then for any given positive parameter $\varepsilon$ small enough,~\eqref{quasilinear2}--\eqref{quasi} admits a unique global smooth solution.
\end{thm}
Now consider the Cauchy problem of 2-D nonlinear elastic waves:
\begin{align}\label{elast1}
&\partial^2_t u-c_2^2 \triangle u-(c_1^2-c_2^2)\nabla \nabla\cdot u =N_2(\nabla u, \nabla^2 u)+N_3(\nabla u, \nabla^2 u),\\\label{elast2}
&t=0:~u=\varepsilon f, u_t=\varepsilon g,
\end{align}
where the nonlinearity is given by~\eqref{N2}--\eqref{N4}. We have
\begin{thm}\label{main}
(Main theorem)Consider the Cauchy problem~\eqref{elast1}--\eqref{elast2}.  Assume that~\eqref{elast1} satisfies the first null condition~~$d_1=0$ and
the second null condition~$e_1=e_2=0$, and the initial data~$f,g$ is radial symmetric and smooth and has compact support. Then for any given positive parameter $\varepsilon$ small enough,~\eqref{elast1}--\eqref{elast2} admits a unique global smooth solution.
\end{thm}

\begin{proof}
For the~Cauchy problem~\eqref{elast1}--\eqref{elast2}, by the radial symmetry of initial data~$f,g$, the rotation invariance of the elastic wave equations and the uniqueness of smooth solutions of the Cauchy problem~\eqref{elast1}--\eqref{elast2}, the solution $u$ is also radial symmetric. So~we can write $u$ as
\begin{align}\label{biao}
u(t,x)=x \psi(t,r),~r=|x|,
\end{align}
where $\psi$ is a scalar function. By~\eqref{biao}, it is easy to see that
\begin{align}\label{curl}
(\nabla^{\bot}\cdot u)(t,x)=0.
\end{align}
It follows from Hodge decomposition and \eqref{curl} that
\begin{align}\label{div}
\Delta u=\nabla \nabla\cdot u+\nabla^{\bot}\nabla^{\bot}\cdot u=\nabla \nabla\cdot u.
\end{align}
Thanks to \eqref{div}, for the linear part of \eqref{elast1}, we have
\begin{align}\label{curl11}
\partial^2_t u-c_2^2 \triangle u-(c_1^2-c_2^2)\nabla \nabla\cdot u=\partial^2_t u-c_1^2 \triangle u.
\end{align}
Next we consider the quadratically nonlinear term $N_2(\nabla u,\nabla^2 u)$. Inserting the first null condition $d_1=0$ and~\eqref{curl} into the expressions \eqref{N2} gives
\begin{align}\label{n22}
N_2(\nabla u,\nabla^2 u)=\big(N^{(1)}_2(\nabla u,\nabla^2 u),N^{(2)}_2(\nabla u,\nabla^2 u) \big),
\end{align}
where
\begin{align}\label{n211}
N^{(1)}_2(\nabla u,\nabla^2 u)
=d_3\partial_1\big(Q_{12}( u^{1}, u^{2})\big)+d_3Q_{12}(\nabla\cdot u, u^{2}),
\end{align}
\begin{align}\label{n212}
N^{(2)}_2(\nabla u,\nabla^2 u)
=d_3\partial_2\big(Q_{12}( u^{1}, u^{2})\big)+d_3Q_{12}(u^{1}, \nabla\cdot u).
\end{align}
At last we consider the cubically nonlinear term $N_3(\nabla u,\nabla^2 u)$.
 If we plug the second null condition~$e_1=0$ and~\eqref{curl} back into the expression \eqref{N3}, we obtain
 \begin{align}\label{n23}
N_3(\nabla u,\nabla^2 u)=\big(N^{(1)}_3(\nabla u,\nabla^2 u),N^{(2)}_3(\nabla u,\nabla^2 u) \big),
\end{align}
where
\begin{align}\label{n311}
N^{(1)}_3(\nabla u,\nabla^2 u)&=2e_4Q_{12}\big(Q_{12}(u^1,u^2), u^2\big)+2e_5\partial_1\big((\nabla\cdot u)Q_{12}(u^1,u^2)\big)\nonumber\\
&+e_5Q_{12}\big((\nabla\cdot u)^2, u^{2}\big),
\end{align}
\begin{align}\label{n312}
N^{(2)}_3(\nabla u,\nabla^2 u)&=2e_4Q_{12}\big(u^1,Q_{12}(u^1,u^2)\big)+2e_5\partial_2\big((\nabla\cdot u)Q_{12}(u^1,u^2)\big)\nonumber\\
&+e_5Q_{12}\big(u^{1},(\nabla\cdot u)^2 \big).
\end{align}
According to~\eqref{curl11}--\eqref{n23}, we see that the~Cauchy problem~\eqref{elast1}--\eqref{elast2} reduces to
\begin{align}\label{reduce}
&\partial^2_t u^i-c_1^2 \triangle u^i=\widetilde{g}^{ijk}_{lmn}\partial_{l}(\partial_{m} u^{j}\partial_{n} u^{k})+\widetilde{h}^{ijkp}_{lmnq}\partial_{l}(\partial_{m}u^{j}\partial_{n}u^{k}\partial_{q} u^{p}),\\\label{reduce2}
&t=0:~u^i=\varepsilon f^i, u^i_t=\varepsilon g^i,~1\leq i\leq 2,
\end{align}
where
\begin{align}
\widetilde{g}^{ijk}_{lmn}\partial_{l}(\partial_{m} u^{j}\partial_{n} u^{k})&=N^{(i)}_2(\nabla u,\nabla^2 u),~1\leq i\leq 2,\\
\widetilde{h}^{ijkp}_{lmnq}\partial_{l}(\partial_{m}u^{j}\partial_{n}u^{k}\partial_{q} u^{p})&=N^{(i)}_3(\nabla u,\nabla^2 u),~1\leq i\leq 2.
\end{align}
It is to easy to verify that $\{\widetilde{g}^{ijk}_{lmn}\}, \{\widetilde{h}^{ijkp}_{lmnq}\}$ satisfies the symmetry conditions \eqref{sym1}. \footnote{ The symmetry
can be obtained by \eqref{sym1232} directly.}
Noting that all nonlinear terms contain the null form $Q_{12}$. It follows from
\eqref{n211} and \eqref{n212} that
\begin{align}
\widetilde{g}^{ijk}_{lmn}\omega_l\omega_m\omega_n=0,~\forall~1\leq i,j,k\leq 2, \omega\in S^1.
\end{align}
So~\eqref{reduce} satisfies the first null condition~\eqref{null1}.
According to~\eqref{n311} and \eqref{n312}, we can get
\begin{align}\label{null2}
\widetilde{h}^{ijkp}_{lmnq}\omega_l\omega_m\omega_n\omega_q=0,~\forall~1\leq i,j,k,p\leq 2, \omega\in S^1.
\end{align}
So~\eqref{reduce} satisfies the second null condition~\eqref{null222}. Then Theorem \ref{main} is a corollary of Theorem \ref{thm41}.
\end{proof}
\begin{rem}
In view of the expression \eqref{N3} and~\eqref{curl}, we know that the assumption~$e_2=0$ in Theorem 4.2 is not necessary.
\end{rem}
\section{Discussion}
Some remarks are given as follows.
\begin{rem}
In \cite{Alinhac01}, for 2-D quasilinear wave equations, Alinhac proved that if only the first null condition is satisfied,
then the smooth solution's lifespan~$T_{\varepsilon}\geq \exp(\frac{c}{\varepsilon^2})$, where~$c$ is a constant independent of~$\varepsilon$.
So for the Cauchy problem~\eqref{elast1}--\eqref{elast2},
if~\eqref{elast1} only satisfies the first null condition $d_1=0$ and the initial data is radial symmetric, then it admits the same lifespan estimate, which can be proved by the same method as employed in Theorem \ref{main}.
\end{rem}

\begin{rem}
In the radial symmetric case, there exists only the pressure wave. In the opposite side, when the material is incompressible, there exists only  the shear wave. For 3-D incompressible materials, the global existence of smooth solutions with small data was showed in~Sideris and~Thomases \cite{Sideris05, Sideris07}(see also an alternative proof in Lei and Wang \cite{MR3317810}). For the 2-D incompressible and Hookean type materials,
Lei \cite{Lei14} proved the global existence of smooth solutions with small data by the vector fields method in the Lagrangian coordinates formulation(see also previous almost global existence result in Lei, Sideris and Zhou \cite{MR3391913}  in the Euler coordinates formulation). Then Wang \cite{Wang14} also established the global existence result by different approach and from the point of view
in frequency space in the Euler coordinates formulation.
\end{rem}
\begin{rem}
For the Cauchy problem of 2-D nonlinear elastic waves \eqref{elast1}--\eqref{elast2}, if the initial data $f,g$ is not radial symmetric, the global existence of smooth solutions under the first null condition ~$d_1=0$  and the second null condition $e_1=e_2=0$ is still open. The main difficulty lies in the control of the nonlinear
interaction of fast pressure wave and slow shear wave at the quadratic level.
\end{rem}
\section*{Acknowledgements}
The author would like to express his sincere gratitude to Prof. Zhen Lei and Prof. Yi Zhou for their helpful suggestions and encouragements.

%
%

\end{document}